\documentclass[11pt,reqno,oneside,draft]{amsart}

\usepackage{amsfonts,amssymb,latexsym}  
\usepackage{amsmath}  

\usepackage[authoryear]{natbib}

\numberwithin{equation}{section}

\newcommand{\N}{\ensuremath{\mathbb{N}}}    
     
\newcommand{\Z}{\ensuremath{\mathbb{Z}}}    
\renewcommand{\P}{\ensuremath{\mathbb{P}}}

\newcommand{\E}{\ensuremath{\mathbb{E}}}

\newcommand{\Id}{\mathbf {1}} 

\newcommand{\ee}{e^{\frac1e}}
\newcommand{\given}{\ \vert \ }

\newcommand{\suf}{{\text{suf}}}
\renewcommand{\tau}{\mathcal T}

\newcommand{\MM}{\mathcal{M}}

\numberwithin{equation}{section}

\theoremstyle{plain}
\newtheorem{theorem}{Theorem}
\newtheorem{lemma}[equation]{Lemma}
\newtheorem{corollary}{Corollary}

\theoremstyle{definition}
\newtheorem{definition}[equation]{Definition}
\theoremstyle{remark}

\begin{document}


\subjclass[2000]{}


\title[Random perturbations of stochastic chains]{Random perturbations of stochastic chains with
  unbounded variable length memory}
  
\author{Pierre Collet }
\address{
  Centre de Physique Th\'eorique, CNRS UMR 7644, Ecole Polytechnique,
  91128 Palaiseau Cedex, France} 
 \email{collet@cpht.polytechnique.fr}

\author{Antonio Galves} 
\address{
  Instituto de Matem\'atica e Estat\'{\i}stica, Universidade de S\~ao
  Paulo, BP 66281, 05315-970 S\~ao Paulo, Brasil} 
  \email{galves@ime.usp.br}

\author{Florencia G. Leonardi}
\address{
  Instituto de Matem\'atica e Estat\'{\i}stica, Universidade de S\~ao
  Paulo, BP 66281, 05315-970 S\~ao Paulo, Brasil} 
  \email{leonardi@ime.usp.br}

\date{July 10, 2007} \thanks{
This work is part of PRONEX/FAPESP's project
  \emph{Stochastic behavior, critical phenomena and rhythmic pattern
    identification in natural languages} (grant number 03/09930-9),
  CNPq's project \emph{Stochastic modeling of speech} (grant number
  475177/2004-5) and CNRS-FAPESP project \emph{Probabilistic phonology
    of rhythm}. AG is partially supported by a CNPq fellowship (grant
  308656/2005-9) and FGL is partially supported by a FAPESP fellowship
  (grant 06/56980-0)
  }
  
\begin{abstract}
We consider binary infinite order stochastic chains perturbed by a
  random noise. This means that at each time step, the value assumed
  by the chain can be randomly and independently flipped with a small
  fixed probability. We show that the transition probabilities of the
  perturbed chain are uniformly close to the corresponding transition
  probabilities of the original chain. As a consequence, in
  the case of stochastic chains with unbounded but otherwise finite
  variable length memory, we show that it is possible to
  recover the context tree of the original chain, using a suitable
  version of the algorithm Context, provided that the noise is small
  enough.
\end{abstract}

\maketitle


\section{Introduction}


The original motivation of this paper is the following question. Is it
possible to recover the context tree of a variable length Markov chain
from a noisy sample of the chain. We recall that in a variable length
Markov chain the conditional probability of the next symbol, given the
past depends on a variable portion of the past whose length depends on
the past itself.  This class of models were first introduced by
\citet{rissanen1983} who called them {\sl finite memory sources} or
{\sl tree machines}. They recently became popular in the statistics
literature under the name of {\sl variable length Markov chains
  (VLMC)} \citep{buhlmann1999}.

The notion of variable memory model can be naturally extended to a non
markovian situation where the contexts are still finite, but their
lengths are no longer bounded (see for example \citet{ferrari2003},
\citet{csiszar2006} and \citet{duarte2006}). This leads us to consider
not only randomly perturbed unbounded variable length memory models,
but more generally randomly perturbed infinite order stochastic
chains.

We will consider binary chains of infinite order in which at each time
step the value assumed by the chain can be randomly and independently
flipped with a small fixed probability. Even if the original chain is
markovian, the perturbed chain is in general a chain of infinite order
(we refer the reader to \citet{ffg01} for a self contained introduction to chains of infinite order).  We show that the transition probabilities of the perturbed
chain are uniformly close to the corresponding transition
probabilities of the original chain. More precisely, we prove that the
difference between the conditional probabilities of the next symbol
given a finite past of any fixed length is uniformly bounded above by
the probability of flipping, multiplied by a fixed constant. This is
the content of our first theorem.

Using this result we are able to solve our original problem of
recovering the context tree of a chain with unbounded variable length
from a noisy sample.  To make this point clear, we must explain the
notion of \emph{context}. In his original 1983 paper, Rissanen used the word 
context to designate the minimal suffix of the string
of past symbols which is enough to define the probability of the
next symbol.  Rissanen also observed that this notion is interesting
only if the set of all contexts satisfies the suffix property, which
means that no context is a proper suffix of another context. This
property allows to represent the set of all contexts as the set of
leaves of a rooted labeled tree, henceforth called the \emph{context
  tree} of the chain.  With this representation the process is
described by the tree of all contexts and an associated family of
probability measures on $A$, indexed by the leaves of the tree. Given a
context, its associated probability measure gives the probability of
the next symbol for any past having this context as a suffix.

\citet{rissanen1983} not only introduced the class of variable memory
models but he also introduced the \emph{algorithm Context} to estimate
both the context tree and the associated family of probability transition. 
The way the algorithm Context works
can be summarized as follows. Given a sample produced by a chain with
variable memory, we start with a maximal tree of candidate contexts
for the sample. The branches of this first tree are then pruned until
we obtain a minimal tree of contexts well adapted to the sample. 

From \citet{rissanen1983} to \citet{galves2006}, passing by
\citet{ron1996} and \citet{buhlmann1999}, several variants of the
algorithm Context have been presented in the literature. In all the
variants the decision to prune a branch is taken by considering a {\em
  gain} function. A branch is pruned if the gain function assumes a
value smaller than a given threshold. The estimated context tree is
the smallest tree satisfying this condition. The estimated family of
probability transitions is the one associated to the minimal tree of
contexts.

\citet{rissanen1983} proved the weak consistency of the algorithm
Context when the tree of contexts is finite. \citet{buhlmann1999}
proved the weak consistency of the algorithm also in the finite case
without assuming a prior known bound on the maximal length of the
memory but instead using a bound allowed to grow with the size of the
sample. In both papers the gain function is defined using the log
likelihood ratio test to compare two candidate trees and the main
ingredient of the consistency proofs was the chi-square approximation
to the log likelihood ratio test for Markov chains of fixed order. 

The unbounded case was considered by \citet{ferrari2003},
\citet{duarte2006}, \citet{csiszar2006} and \citet{leonardi2007}. The
first two papers essentially extend to the unbounded case the
original chi-square approach introduced by Rissanen. Instead of the chi-square,
the last
two papers use penalized likelihood algorithms, 
related to the Bayesian Information Criterion (BIC), to
estimate the context tree. We refer the
reader to \citet{csiszar2006} for a nice description of other
approaches and results in this field, including the context tree
maximizing algorithm by \citet{willems1995}.

In the present paper we use a variant of the algorithm Context
introduced in \citet{galves2006} for finite trees and extended to
unbounded trees in \citet{galves2007}. In this variant, the
decision of pruning a branch is taken by considering the difference
between the estimated conditional probabilities of the original branch
and the pruned one, using a suitable threshold. Using exponential
inequalities for the estimated transition probabilities associated to
the candidate contexts, these papers not only show the consistency of
this variant of the algorithm Context, but also provide an
exponential upper bound for the rate of convergence.

This version of the algorithm Context does not distinguish transition
probabilities which are closer than the threshold level used in the
pruning decision. Our first theorem assures that this is what happens
between the conditional probabilities of the original variable memory
chain and the perturbed one, if the probability of random flipping is
small enough. Hence it is natural to expect that with this version of
the algorithm Context, one should be able to retrieve the original
context tree out from the noisy sample. This is actually the case, as
we prove in the second theorem.

The paper is organized as follows. In section~\ref{def} we give the
definitions and state the main results. Section 3 and 4 are devoted to the proof 
of Theorem~1 and 2, respectively.


\section{Definitions and results}\label{def}


Let $A$ denote the binary alphabet $\{0,1\}$ with size $|A|=2$.
Given two integers $m\leq n$, we will denote by $w_m^n$ the sequence
$(w_m, \ldots, w_n)$ of symbols in $A$. The length of the sequence
$w_m^n$ is denoted by $\ell(w_m^n)$ and is defined by $\ell(w_m^n) =
n-m+1$.  Any sequence $w_m^n$ with $m > n$ represents the empty string
and is denoted by $\lambda$. The length of the empty string is
$\ell(\lambda) = 0$.

Given two sequences $w$ and $v$, we will denote by
$vw$ the sequence of length $\ell(v) + \ell(w) $ obtained by
concatenating the two strings.  In particular, $\lambda w = w\lambda=
w$.  The concatenation of sequences is also extended to the case in
which $v$ denotes a semi-infinite sequence, that is
$v=v_{-\infty}^{-1}$.

We say that the sequence $s$ is a \emph{suffix} of the sequence $w$ if
there exists a sequence $u$, with $\ell(u)\geq 1$, such that $w = us$.
In this case we write $s\prec w$. When $s\prec w$ or $s=w$ we write
$s\preceq w$. Given a sequence $w$ we denote by $\suf(w)$ the largest suffix of $w$.

In the sequel $A^j$ will denote the set of all sequences of length $j$
over $A$ and $A^*$ represents the set of all finite sequences, that is
\begin{equation*}
  A^* \,=\, \bigcup_{j=1}^{\infty}\,A^j.
\end{equation*}

\begin{definition} 
  A countable subset $\tau$ of $A^*$ is a \emph{tree} if no sequence
  $s \in \tau$ is a suffix of another sequence $w \in \tau$. This
  property is called the \emph{suffix property}.
\end{definition}

We define the \emph{height} of the tree $\tau$ as
\begin{equation*}
  \ell(\tau) = \sup\{\ell(w) : w\in\tau\}. 
\end{equation*}

In the case $\ell(\tau)<+\infty$ it follows that $\tau$ has finite
cardinality. In this case we say that $\tau$ is \emph{bounded}
and we will denote by $|\tau|$ the number of sequences in $\tau$. On
the other hand, if $\ell(\tau)=+\infty$ then $\tau$ has a countable
number of sequences. In this case we say that the tree $\tau$ is
\emph{unbounded}.
 
Given a tree $\tau$ and an integer $K$ we will denote by $\tau|_K$ the
tree $\tau$ \emph{truncated} to level $K$, that is
\begin{equation*}
  \tau|_K = \{w \in \tau\colon \ell(w) \le K\} \cup \{ w\colon \ell(w)=K 
  \text{ and } w\prec u, \text{ for some } u\in\tau\}.
\end{equation*}

We will say that a tree is \emph{irreducible} if no sequence can be
replaced by a suffix without violating the suffix property. This
notion was introduced in \cite{csiszar2006} and generalizes the
concept of complete tree. 
 
\begin{definition}\label{def:pct}
 A \emph{probabilistic context tree over $A$} is an
  ordered pair $(\tau,p)$ such that
\begin{enumerate}
\item $\tau$ is an irreducible tree;
\item $p = \{p(\cdot|w); w \in \tau\}$ is a family of transition
  probabilities over $A$.
\end{enumerate}
\end{definition}

Consider a stationary stochastic chain $\{X_t\colon t\in\Z\}$
over $A$. Given a sequence $w\in A^j$ we denote by 
\begin{equation*}
 p(w) \,=\,  \P(X_1^j = w)
\end{equation*}
the stationary probability of the cylinder defined by the sequence $w$.
If $p(w) > 0$ we write
\begin{equation*}
p(a|w) \,=\, \P ( X_0 =a \given X_{-j}^{-1}=w)\,.
\end{equation*}

\begin{definition}
  A sequence $w\in A^j$ is a \emph{context} for the process $\{X_t\colon t\in\Z\}$ if
  $p(w)>0$ and for any semi-infinite sequence $x_{-\infty}^{-1}$ such
  that $w$ is a suffix of $x_{-\infty}^{-1}$ we have that
\begin{equation}\label{eq:pt}
  \P ( X_0 =a \given X_{-\infty}^{-1}=x_{-\infty}^{-1}) \,=\, p(a|w),
  \quad\text{for all $a\in A$},
\end{equation}
and no suffix of $w$ satisfies this equation.
\end{definition}

\begin{definition}
We say that the process $\{X_t\colon t\in\Z\}$ is \emph{compatible} with the probabilistic context tree $(\tau,\bar p)$ if the following conditions are satisfied
\begin{enumerate}
\item $w\in\tau$ if and only if $w$ is a context for the process $\{X_t\colon t\in\Z\}$. 
\item For any $w\in\tau$ and any $a\in A$, $\bar p(a|w) =  \P ( X_0 =a \given X_{-|w|}^{-1}=w)$.
\end{enumerate}
\end{definition}

In the unbounded case, the compactness of $A^{\Z}$ assures that there is at least one
stationary stochastic chain compatible with a probabilistic context tree. The uniqueness requires 
further conditions, such as the ones presented in \citet{fernandez2002}.

\begin{definition}
  A probabilistic context tree $(\tau,p)$ is of \emph{type B} if it satisfies the following conditions
  \begin{enumerate}
  \item {\em Non-nullness}, that is
  \[
 \alpha :=  \inf_{w\in\tau,a\in A} p(a|w)  > 0;
  \]
  \item {\em Log-continuity}, that is
  \[
  \beta_k \rightarrow 0 \text{ when } k\rightarrow\infty,
  \]
   where the sequence 
  $\{\beta_k\}_{k\in\N}$ is defined by 
    \begin{equation*}\label{continuity}
      \beta_k \,  :=  \,\sup \{ | 1 - \textstyle{\frac{p(a | w)}{p(a | v)}}|\colon a\in A, v, w \in \tau \, \mbox{ with } \,
      w \overset{k}{=} v \}.
    \end{equation*}  
    Here, $w \overset{k}{=} v$ means that there exists a sequence $u$,
    with $\ell(u) = k$ such that $u\prec w$ and $u\prec v$. The sequence $\{\beta_k\}_{k\in\N}$ is called 
   the \emph{continuity rate}.
  \end{enumerate}
 \end{definition} 
  
  For a probabilistic context tree of type B with summable continuity rate, the maximal coupling argument used in \citet{fernandez2002} implies the uniqueness of the law of the chain consistent with it. Then, we will 
  assume here that the continuity rate is summable, that is  
  \begin{equation}\label{beta}
  \beta := \sum_{k\in\N} \beta_k \,<\, +\infty.
  \end{equation}

This condition immediately  implies that 
\begin{equation*}
1\; \leq \;\;\beta^* := \prod_{k=0}^{+\infty} (1+\beta_k)\;\; <\; +\infty.
\end{equation*}

Given an integer $k\geq 1$ we define
\begin{equation}
D_k = \min_{w\in\tau: \ell(w)\leq k}\; \max_{a\in A}\{\,|p(a|w) - p(a|\suf(w))|\,\}.
\end{equation}


In this paper we are interested on the effect of a  Bernoulli
noise flipping independent from the successive symbols produced by the
chain. Namely, let $\{\xi_{t}\colon t\in \Z\}$ be an i.i.d. sequence of random
variables taking values in the alphabet $A$, independent of $\{X_{t}\colon t\in \Z\}$,
with 
$$
\P\big(\xi_{t}=0)=1-\epsilon,
$$
where $\epsilon$ is a fixed noise parameter in $[0,1]$.
For $a$ and $b$ in $A$, we define
$$
a\oplus b=a+b\pmod2,
$$
and $\bar a=1\oplus a$.
We now define the stochastically perturbed chain
$\{Z_t\colon t\in\Z\}$ by
$$
Z_t=X_{t}\oplus\xi_{t}.
$$ 
The process $\{Z_t\colon t\in\Z\}$ is an example of a hidden Markov model. In the case
$\epsilon=1/2$, $\{Z_t\colon t\in\Z\}$ is an i.i.d. uniform Bernoulli. However, in
general it is not a chain of finite order.

We will use the shorthand notation
$$
q(w_1^j)=\P\big(Z_1^j=w_1^j\big)
$$
and
$$
q(a|w_{-j}^{-1})=\P\big(Z_0=a\given
Z_{-j}^{-1}=w_{-j}^{-1}\big)
$$
to denote the probabilities corresponding to the process $\{Z_t\colon t\in\Z\}$.  
We also define
\begin{equation}\label{ek}
q_k = \min\{\,q(w)\colon \ell(w)\leq k \text{ and } q(w) > 0\,\}.
\end{equation}

We can now state our first theorem.

\begin{theorem}\label{principal} Assume the chain
  $\{X_t\colon t\in\Z\}$ has summable continuity rate. Then, for any
$\epsilon\in [0,1]$  and for any $j\geq0$
\[
\sup_{w_{-j}^{0}}\big| \P\big(Z_0=w_0\,\big|\,
\,Z_{-j}^{-1}=\,w_{-j}^{-1}\big)-
\P\big(X_0=w_0\,\big|\,
\,X_{-j}^{-1}=\,w_{-j}^{-1}\big) \big|\;\le\; (1+\frac{4\beta\beta^*}{\alpha})\,\epsilon\,.
\]
\end{theorem}

To state the second theorem we first need to present the version of the Algorithm Context introduced in 
\citet{galves2006} and \citet{galves2007}.


In what follows we will assume that $z_1, z_2 \dotsc, z_n$ is a
sample of the observed chain $\{Z_t\colon t\in\Z\}$ and that the underlying chain
$\{X_t\colon t\in \Z\}$ is compatible with the probabilistic context tree $(\tau,p)$. 

 For any finite string $w$ with $\ell(w) \le n$, we denote by
 $N_n(w)$ the number of occurrences of the string in the sample; that is
  \begin{equation}
    \label{eq:Nn}
     N_n(w)=\sum_{t=0}^{n-\ell(w)}\Id\{z_{t+1}^{t+\ell(w)}=w\}.
  \end{equation}
  
  For any element $a \in A$ and any finite sequence $w\in A^*$, the empirical
  transition probability $\hat{q}_n(a|w)$ is defined by
  \begin{equation}\label{phat}
  \hat{q}_n(a|w)= \frac{N_n(wa)+1}{N_n(w\cdot)+|A|} \/.
  \end{equation}
  where 
  \[
  N_n(w\cdot)=\sum_{b \in A} N_n(wb)\,.
  \]
  This definition of $\hat{q}_n(a|w)$ is convenient because it is
  asymptotically equivalent to $\frac{N_n(wa)}{N_n(w\cdot)}$ and it
  avoids an extra definition in the case $N_n(w\cdot)=0$.

  The variant of Rissanen's {\sl algorithm Context} we will use is defined as follows.
  First of all, let us define for any finite string $w\in A^*$:
\[ 
\Delta_n(w) = \max_{a\in A}
|\hat{q}_n(a|w)-\hat{q}_n(a|\suf(w))| \/.
\]
The $\Delta_n(w)$ operator computes a distance between the empirical transition probabilities 
associated to the sequence  $w$ and the one associated to the sequence $\suf(w)$.

\begin{definition}
  Given $\delta > 0$ and $d < n$, the tree estimated with the algorithm Context is
  \[
  \hat{\tau}_n^{\delta,d} = \{w\in A_1^d: \Delta_n(w) >
  \delta\,\text{ and }\,\Delta_n(uw)\leq \delta \;\;\forall u\in A_1^{d-\ell(w)}\},
  \]
  where $A_1^r$ denotes the set of all sequences of length at most $r$. In the case 
  $\ell(w) = d$ we have $A_1^{d-\ell(w)} = \emptyset$.
\end{definition}

It is easy to see that $\hat{\tau}_n^{\delta,d}$ is a tree. Moreover, the way we defined 
$\hat{q}_n(\cdot|\cdot)$  in (\ref{phat})
associates a probability distribution to each sequence in $\hat{\tau}_n^{\delta,d}$.

We may now state our second theorem.

\begin{theorem}\label{expobounded}
 Let $K$ be an integer and let $z_1, z_2 \dotsc, z_n$ be a sample of the perturbed chain $\{Z_t\colon t\in\Z\}$.
 Then, for any $d$ satisfying 
 \begin{equation}\label{d}
 d >  \max_{w\in \tau|_K} \min\,\{ \ell(v)\colon v\in\tau,  v\succeq w\},
 \end{equation}
 for any 
 $\delta$ such that $2(1+\frac{4\beta\beta^*}{\alpha})\epsilon <\delta < D_d-2(1+\frac{4\beta\beta^*}{\alpha})\epsilon$ and for any $n$ such that
 \[
 n \;>\; \frac{4(|A|+1)}{[\min(\delta, D_d-\delta)-2\epsilon(1+\frac{4\beta\beta^*}{\alpha})]q_d
 } + d
 \] 
 we  have 
  \begin{equation*} 
  \P(\hat\tau_n^{\delta,d}|_K \neq \tau|_K)\; \leq \;
  4\,\ee\,(|A|+1)\,|A|^{d+1} \exp \bigl[- (n-d)\;\frac{[ \min(\delta, D_d-\delta)-2\epsilon(1+\frac{4\beta\beta^*}{\alpha})]^2 q_d^2}{256e(1+\beta)|A|^2(d+1)}\bigl]. 
    \end{equation*}
\end{theorem}

As a consequence we obtain the following strong consistency result.

\begin{corollary}\label{main_cor}
 For any integer $K$ and for almost all infinite sample $z_1,z_2\dotsc$ 
 there exists a $\bar n$ such that, for any 
  $n\geq \bar n$ we have
  \begin{equation}
  \hat\tau_n^{\delta,d}|_K = \tau|_K,
  \end{equation}
  where $d$ is given by (\ref{d}) and $\delta$ is such that $2(1+\frac{4\beta\beta^*}{\alpha})\epsilon <\delta < D_d-2(1+\frac{4\beta\beta^*}{\alpha})\epsilon$.
  \end{corollary}

 
\section{Proof of Theorem~\ref{principal}}


We start by proving three preparatory lemmas.

\begin{lemma}\label{lescon2}
For any $k>j\ge 0$ and any   $\epsilon\in [0,1]$  we have
\begin{align*}
\sup_{w_{-\infty}^{0},a,b}\; \bigl| \P\big(X_0=w_0&\given  
X_{-j}^{-1}=w_{-j}^{-1},X_{-j-1}=a,Z_{-j-1}=b,
Z_{-k}^{-j-2}=w_{-k}^{-j-2}\big) - p\big(w_0\given
w_{-\infty}^{-1}\big)\bigr|\\
& \leq\;\; \beta_j.
\end{align*}
\end{lemma}

\begin{proof}
We observe that for $j\geq0$ it follows from the independence of the flipping procedure that
\begin{align*}
\P\big(X_0=&w_0\,\big|\,
\,X_{-j}^{-1}=\,w_{-j}^{-1},X_{-j-1}=a,\,Z_{-j-1}=b,\,
 Z_{-k}^{-j-2}=\,w_{-k}^{-j-2}\big)
\\
&= \; \frac{\sum_{u_{-k}^{-j-2}}\P\big(X_{-k}^{0}=u_{-k}^{-j-2}a\, w_{-j}^{-1}w_0\big)
\P\big(Z_{-k}^{-j-1}=w_{-k}^{-j-2}b\given
X_{-k}^{-j-1}=u_{-k}^{-j-2}a\big)
}
 {\sum_{u_{-k}^{-j-2}}\P\big(X_{-k}^{-1}= u_{-k}^{-j-2}a\, w_{-j}^{-1}\big)
\P\big(Z_{-k}^{-j-1}=w_{-k}^{-j-2}b\given
X_{-k}^{-j-1}=u_{-k}^{-j-2}a\big)}.
\end{align*}
It is easy to see using conditioning on the infinite past that 
\begin{align*}
\inf_{v_{-\infty}^{-j-1}}\P\big(
X_0=w_0&\given X_{-j}^{-1}=w_{-j}^{-1},
X_{-\infty}^{-j-1}=v_{-\infty}^{-j-1}\big)
\\
&\leq\;\P\big(X_0=w_0\given X_{-k}^{-1}=u_{-k}^{-j-2}a\, w_{-j}^{-1}\big)
\\
&\leq\;\sup_{v_{-\infty}^{-j-1}}\P\big(
X_0=w_0\given X_{-j}^{-1}=w_{-j}^{-1},
X_{-\infty}^{-j-1}=v_{-\infty}^{-j-1}\big).
\end{align*}
Then, using continuity we have
\begin{align*}
p\big(w_0\given
w_{-\infty}^{-1}\big) - \beta_j \;&\leq\;\P\big(X_0=w_0\given X_{-k}^{-1}=u_{-k}^{-j-2}a\, w_{-j}^{-1}\big)
\leq\; p\big(w_0\given
w_{-\infty}^{-1}\big) +\beta_j
\end{align*}
and the assertion of the Lemma  follows immediately.
\end{proof}

\begin{lemma}\label{lescon3}
For any $\epsilon\in [0,1]$ and 
for any $k\ge 0$ we have 
\begin{equation*}
\inf_{w_{-k}^{0}}\P\big(Z_0= w_0\,\big|\,
Z_{-k}^{-1}=\,w_{-k}^{-1}\big)
\;\ge\; \alpha,
\end{equation*}
and
\begin{equation*}
\inf_{w_{-k}^{0}}\P\big(X_0= w_0\,\big|\,
Z_{-k}^{-1}=\,w_{-k}^{-1}\big)
\;\ge\; \alpha.
\end{equation*}
Moreover, for any
$0\le  j\le k$ we have
\begin{equation*}
\inf_{w_{-k}^{-1}}
\P\big(X_{-j-1}=w_{-j-1}\given X_{-j}^{-1}=w_{-j}^{-1},\,
Z_{-k}^{-j-2}=\,w_{-k}^{-j-2}\big)\;\ge\;\frac{\alpha}{\beta^*}.
\end{equation*}
\end{lemma}

\begin{proof}
We first observe that 
\begin{equation*}
\P\big(Z_0= w_0\,\big|\,
Z_{-k}^{-1}=\,w_{-k}^{-1}\big)=(1-\epsilon)\,
\P\big(X_0= w_0\,\big|\, Z_{-k}^{-1}=\,w_{-k}^{-1}\big)
+\epsilon \,\P\big(X_0=\bar w_0\,\big|\, Z_{-k}^{-1}=\,w_{-k}^{-1}\big).
\end{equation*}
It is therefore enough to prove the second assertion. 
From the independence of the flipping procedure we have 
\begin{align*}
&\P\big(X_0= w_0\,\big|\,
Z_{-k}^{-1}=\,w_{-k}^{-1}\big)\; =
\\
&\lim_{l\to\infty}
\frac{(1-\epsilon)^{k}\sum_{u_{-l}^{-1}}
p\big(w_0\,\big|\,u_{-l}^{-1}w_{-\infty}^{-l-1}\big)
\P\big(
X_{-l}^{-1}=\,u_{-l}^{-1}\,\big|\,X_{-\infty}^{-l-1}=
w_{-\infty}^{-l-1}\big)(\epsilon/(1-\epsilon))^{\sum_{j=-k}^{-1}
  u_j\oplus w_j}
}
{(1-\epsilon)^{k}\sum_{u_{-l}^{-1}}
\P\big(
X_{-l}^{-1}=\,u_{-l}^{-1}\,\big|\,X_{-\infty}^{-l-1}=
w_{-\infty}^{-l-1}\big)(\epsilon/(1-\epsilon))^{\sum_{j=-k}^{-1}
  u_j\oplus w_j}
}
\\
&\ge\;\alpha.
\end{align*}
For the last assertion we first observe that
\begin{align*}
\P\big(&X_{-j-1}=w_{-j-1}\given X_{-j}^{-1}=w_{-j}^{-1},\,
Z_{-k}^{-j-2}=\,w_{-k}^{-j-2}\big)\\
& \frac{\sum_{x_{-k}^{-j-2}} 
\P\big(Z_{-k}^{-j-2}=w_{-k}^{-j-2}\given X_{-k}^{-j-2}=x_{-k}^{-j-2}\big)
\P\big(X_{-j-1}^{-1}=w_{-j-1}^{-1}, X_{-k}^{-j-2}=x_{-k}^{-j-2}\big)}
 {\sum_{x_{-k}^{-j-2}} 
\P\big(Z_{-k}^{-j-2}=w_{-k}^{-j-2}\given X_{-k}^{-j-2}=x_{-k}^{-j-2}\big)
\P\big(X_{-j}^{-1}=w_{-j}^{-1}, X_{-k}^{-j-2}=x_{-k}^{-j-2}\big)}.
\end{align*}
Moreover,
\begin{align*}
 &  \frac{
\P\big(X_{-j-1}^{-1}=w_{-j-1}^{-1}, X_{-k}^{-j-2}=x_{-k}^{-j-2}\big)}
 {\P\big(X_{-j}^{-1}=w_{-j}^{-1}, X_{-k}^{-j-2}=x_{-k}^{-j-2}\big)}\;=\\
& \frac{
\prod_{l=1}^{j+1}\P\big(X_{-l}=w_{-l} |\, X_{-j-1}^{-l-1}=w_{-j-1}^{-l-1},X_{-k}^{-j-2}=x_{-k}^{-j-2}\big)
 \prod_{l=j+2}^k\P\big(X_{-l}=w_{-l}|\, X_{-k}^{-l-1}=x_{-k}^{-l-1}\big)
}
 {\prod_{l=1}^j\P\big(X_{-l}=w_{-l} |\, X_{-j}^{-l-1}=w_{-j}^{-l-1},X_{-k}^{-j-2}=x_{-k}^{-j-2}\big)
 \prod_{l=j+2}^k\P\big(X_{-l}=w_{-l}|\, X_{-k}^{-l-1}=x_{-k}^{-l-1}\big)
 }\\
 &\geq\;\P\big(X_{-j-1}=w_{-j-1} |\,X_{-k}^{-j-2}=x_{-k}^{-j-2}\big)  \;\prod_{l=1}^j\frac{
   \P\big(X_{-l}=w_{-l} |\, X_{-j-1}^{-l-1}=w_{-j-1}^{-l-1},X_{-k}^{-j-2}=x_{-k}^{-j-2}\big) }
 {\P\big(X_{-l}=w_{-l} |\, X_{-j}^{-l-1}=w_{-j}^{-l-1},X_{-k}^{-j-2}=x_{-k}^{-j-2}\big)
 }
\end{align*} 
 and using non-nullness and log-continuity this is bounded below by  
 \[
  \alpha \; \prod_{l=1}^j\frac{1}{1+ \beta_{j-l}} \;\geq\; \frac{\alpha}{\beta^*}.
\]
This finishes the proof of the Lemma.
\end{proof}

\begin{lemma}\label{lescon4}
For any $k>j\ge 0$ and any $\epsilon\in [0,1]$
\begin{equation*}
\sup_{w_{-k}^{0}}\;\P\big(X_{-j-1}=\bar w_{-j-1}\,\big|\,
\,X_{-j}^{-1}=\,w_{-j}^{-1},\,
Z_{-k}^{-j-1}=\,w_{-k}^{-j-1}\big)
\;\le\;  \frac{\beta^*}{\alpha}\;\epsilon\,.
\end{equation*}
\end{lemma}

\begin{proof}
We have
\begin{align*}
\P\big(&X_{-j-1}=\bar w_{-j-1}\given
X_{-j}^{-1}=\,w_{-j}^{-1},\,
Z_{-k}^{-j-1}=\,w_{-k}^{-j-1}\big)
\\
& = \;\frac{
\P\big(X_{-j-1}=\bar w_{-j-1},Z_{-j-1}= w_{-j-1}\given
X_{-j}^{-1}=w_{-j}^{-1},Z_{-k}^{-j-2}=\,w_{-k}^{-j-2}\big)}
{
\P\big(Z_{-j-1}= w_{-j-1}\given
X_{-j}^{-1}=w_{-j}^{-1},Z_{-k}^{-j-2}=w_{-k}^{-j-2}\big)
}\\
& = \;
\frac{\epsilon\; \P\big(X_{-j-1}=\bar w_{-j-1}\given
X_{-j}^{-1}=w_{-j}^{-1},Z_{-k}^{-j-2}=\,w_{-k}^{-j-2}\big)}
{
\P\big(Z_{-j-1}= w_{-j-1}\given
X_{-j}^{-1}=w_{-j}^{-1},Z_{-k}^{-j-2}=w_{-k}^{-j-2}\big)
}.
\end{align*}
It follows from Lemma~\ref{lescon3} that
\begin{align*}
\P\big(Z_{-j-1}= w_{-j-1}&\given
X_{-j}^{-1}=w_{-j}^{-1},Z_{-k}^{-j-2}=w_{-k}^{-j-2}\big)
\\
& = \;(1-\epsilon)\;
\P\big(X_{-j-1}= w_{-j-1}\given
X_{-j}^{-1}=w_{-j}^{-1},Z_{-k}^{-j-2}=w_{-k}^{-j-2}\big)
\\
&\quad+
\epsilon\;
\P\big(X_{-j-1}= \bar w_{-j-1}\given
X_{-j}^{-1}=w_{-j}^{-1},Z_{-k}^{-j-2}=w_{-k}^{-j-2}\big)
\\
&\geq\; \frac{\alpha}{\beta^*}
\end{align*}
This concludes the proof of Lemma~\ref{lescon4}.
\end{proof}

\begin{proof}[Proof of Theorem \ref{principal}] 
   We first observe that
  \[
  \P\big(Z_0=w_0\,\big|\, Z_{-k}^{-1}= w_{-k}^{-1}\big)
  =(1-\epsilon)\,\P\big(X_0=w_0\,\big|\, Z_{-k}^{-1}=
  w_{-k}^{-1}\big)+\epsilon \, \P\big(X_0=\bar w_0\,\big|\, Z_{-k}^{-1}=
  w_{-k}^{-1}\big).
  \]
  Therefore,
  \[
  \big|\P\big(Z_0=w_0\,\big|\, Z_{-k}^{-1}= w_{-k}^{-1}\big) -
  \P\big(X_0=w_0\,\big|\, Z_{-k}^{-1}= w_{-k}^{-1}\big)\big|\;\le\;\epsilon
  \]
  and if $k=0$ the Theorem is proved. We will now assume $k\ge1$ and we 
  write
  \begin{align*}
  \P\big(X_0=w_0\,\big|\, Z_{-k}^{-1}=& w_{-k}^{-1}\big)\,-\,
  \P\big(X_0=w_0\,\big|\, X_{-k}^{-1}=
  w_{-k}^{-1}\big)\\
 &  = \quad\sum_{j=0}^{k-1}
  \big[\P\big(X_0=w_0\,\big|\, X_{-j}^{-1}= w_{-j}^{-1},\, 
  Z_{-k}^{-j-1}= w_{-k}^{-j-1}\big)\\
  & \qquad\qquad- \;\P\big(X_0=w_0\,\big|\, X_{-j-1}^{-1}= w_{-j-1}^{-1},\, 
  Z_{-k}^{-j-2}= w_{-k}^{-j-2}\big)
  \big].
  \end{align*}
We will bound each term in the sum separately. We can write
\begin{align*}
\P\big(X_0=&w_0\given  X_{-j}^{-1}= w_{-j}^{-1},\,
Z_{-k}^{-j-1}= w_{-k}^{-j-1}\big)-
\P\big(X_0=w_0\given X_{-j-1}^{-1}= w_{-j-1}^{-1},\, 
Z_{-k}^{-j-2}= w_{-k}^{-j-2}\big)
\\[.2cm]
&=\;\sum_{b\in\{0,1\}}\big[
\P\big(X_0=w_0 \given X_{-j}^{-1}= w_{-j}^{-1},\, X_{-j-1}=b,\,
Z_{-k}^{-j-1}= w_{-k}^{-j-1}\big)
\\
&\mspace{110mu}-\;\P\big(X_0=w_0 \given X_{-j-1}^{-1}= w_{-j-1}^{-1},\, 
Z_{-k}^{-j-2}= w_{-k}^{-j-2}\big)\big]\\[.2cm]
&\qquad\times \P\big(X_{-j-1}=b \given X_{-j}^{-1}= w_{-j}^{-1},\, 
Z_{-k}^{-j-1}= w_{-k}^{-j-1}\big).
\end{align*}
The above sum is a sum of two terms, one with $b=\bar w_{-j-1}$, the
other one with $b= w_{-j-1}$. We will bound above these two terms
separately. For the first term we have the bound
\begin{align*}
\big|
\P\big(X_0=&w_0\given X_{-j}^{-1}= w_{-j}^{-1},\, X_{-j-1}=
\bar w_{-j-1},\,
Z_{-k}^{-j-1}= w_{-k}^{-j-1}\big)
\\[.1cm]
&\mspace{70mu}-\;
\P\big(X_0=w_0\,\big|\, X_{-j-1}^{-1}= w_{-j-1}^{-1},\, 
Z_{-k}^{-j-2}= w_{-k}^{-j-2}\big)\big|
\\[.1cm]
&\times \;
\P\big(X_{-j-1}=\bar w_{-j-1}\,\big|\, X_{-j}^{-1}= w_{-j}^{-1},\, 
Z_{-k}^{-j-1}= w_{-k}^{-j-1}\big)\;\le\; \frac{2\,\beta_j\beta^*}{\alpha}\,\epsilon 
\end{align*}
from Lemma~\ref{lescon2} and Lemma~\ref{lescon4}.
 For the other term we can write
\begin{align*}
&\big|
\P\big(X_0=w_0\given X_{-j}^{-1}= w_{-j}^{-1},\, X_{-j-1}=
 w_{-j-1},\,
Z_{-k}^{-j-1}= w_{-k}^{-j-1}\big)
\\
&\mspace{100mu}-\;
\P\big(X_0=w_0\,\big|\, X_{-j-1}^{-1}= w_{-j-1}^{-1},\, 
Z_{-k}^{-j-2}= w_{-k}^{-j-2}\big)\big|
\\
&\qquad\times \;
\P\big(X_{-j-1}= w_{-j-1}\,\big|\, X_{-j}^{-1}= w_{-j}^{-1},\, 
Z_{-k}^{-j-1}= w_{-k}^{-j-1}\big)
\\
&\le \sum_{a\in\{0,1\}}\big|
\P\big(X_0=w_0\,\big|\, X_{-j}^{-1}= w_{-j}^{-1},\, X_{-j-1}=
 w_{-j-1},\,
Z_{-k}^{-j-1}= w_{-k}^{-j-1}\big)\\[-.1cm]
&\mspace{100mu}-\;\P\big(X_0=w_0\,\big|\, X_{-j-1}^{-1}= w_{-j-1}^{-1},\, Z_{-j-1}=a,\,
Z_{-k}^{-j-2}= w_{-k}^{-j-2}\big)\big|
\\
&\qquad\times\;
\P\big(Z_{-j-1}=a\,\big|\, X_{-j-1}^{-1}= w_{-j-1}^{-1},\, 
Z_{-k}^{-j-2}= w_{-k}^{-j-2}\big)
\\
&\qquad\times \;
\P\big(X_{-j-1}= w_{-j-1}\,\big|\, X_{-j}^{-1}= w_{-j}^{-1},\, 
Z_{-k}^{-j-1}= w_{-k}^{-j-1}\big).
\end{align*}
Using the fact that the term in the sum with
$a=w_{-j-1}$ vanishes this is bounded above by 
\begin{align*}
\big|
\P\big(X_0=w_0\,&\big|\, X_{-j}^{-1}= w_{-j}^{-1},\, X_{-j-1}=
 w_{-j-1},\,
Z_{-k}^{-j-1}= w_{-k}^{-j-1}\big)
\\
&\mspace{100mu}-\;\P\big(X_0=w_0\,\big|\, X_{-j-1}^{-1}= w_{-j-1}^{-1},\,
Z_{-j-1}=\bar w_{-j-1},\,
Z_{-k}^{-j-2}= w_{-k}^{-j-2}\big)\big|
\\
&\qquad\times\;
\P\big(Z_{-j-1}=\bar w_{-j-1}\,\big|\, X_{-j-1}^{-1}= w_{-j-1}^{-1},\, 
Z_{-k}^{-j-2}= w_{-k}^{-j-2}\big)\;\le\;  2\,\beta_j\,\epsilon 
\end{align*}
from Lemma \ref{lescon2}.
Putting all the above bounds together we get
\begin{align*}
\big|\P\big(Z_0=w_0\,\big|\, Z_{-k}^{-1}= w_{-k}^{-1}\big) -
\P\big(X_0=w_0\,\big|\, X_{-k}^{-1}= w_{-k}^{-1}\big)\big|
\le\; \epsilon+ \frac{2\beta\beta^*}{\alpha} \epsilon + 2\beta\epsilon
\end{align*}
and the Theorem  follows.
\end{proof}


\section{Proof of  Theorem~\ref{expobounded}}


We start by proving four new auxiliary Lemmas.

\begin{lemma}\label{mixing}
  For any $i\geq 1$, any $k > i$, any $j\geq 1$ and any
  finite sequence $w_1^j$, the following inequality holds
  \begin{equation*}
    \sup_{x_1^{i},\,\theta_1^i\in A^i} 
    |\P(Z_{k}^{k+j-1}=w_1^j \given X_1^i=x_1^i,\,\xi_1^i=\theta_1^i)-q(w_1^j)|\;
    \leq \,j\, \beta_{k-i-1}\,.
  \end{equation*}
\end{lemma}

\begin{proof}
Observe that for any $x_1^{i},\,\theta_1^i\in A^i$ 
\begin{align*}
 |\P(&Z_{k}^{k+j-1}=w_1^j \given X_1^i=x_1^i,\,\xi_1^i=\theta_1^i)-q(w_1^j)|\\[.1cm]
 & = \;|  \sum_{x_{k}^{k+j-1}\in A^j} \P(X_k^{k+j-1}=x_k^{k+j-1},Z_{k}^{k+j-1}=w_1^j \given X_1^i=x_1^i,\xi_1^i=\theta_1^i)-q(w_1^j)|\\
 & = \; | \sum_{x_{k}^{k+j-1}\in A^j} 
  \P(Z_{k}^{k+j-1}=w_1^j \given X_k^{k+j-1}=x_k^{k+j-1})
  \P(X_k^{k+j-1}=x_k^{k+j-1} \given X_1^i=x_1^i,\xi_1^i=\theta_1^i)\\
  &\mspace{100mu}-q(w_1^j)|
 \end{align*}
 by the independence of the flipping procedure. The last term can be bounded above by
 \begin{align*}
  & = \;  \sum_{x_{k}^{k+j-1}\in A^j}
   \P(Z_{k}^{k+j-1}=w_1^j \given X_k^{k+j-1}=x_k^{k+j-1})\; |\,
     \P(X_k^{k+j-1}=x_k^{k+j-1} \given X_1^i=x_1^i)\\
     &\mspace{120mu} -  
      \P(X_k^{k+j-1}=x_k^{k+j-1})\,|.
      \end{align*}
      Then, we can use Lemma~3.6 in \citet{galves2007} to bound above the last sum with
      \begin{equation*}
 \sum_{x_{k}^{k+j-1}\in A^j} 
  j \,\beta_{k-i-1} \; \P(X_k^{k+j-1}=x_k^{k+j-1})
 \end{equation*}
 We conclude the proof of Lemma~\ref{mixing}. 
\end{proof}

\begin{lemma}\label{estim1} 
  For any finite sequence $w$ and any $t>0$ we have
  \begin{equation*}
    \P(\,|N_n(w)-(n-\ell(w)+1)q(w)|\,>\,t\,)\,\leq \,e^{\frac1e} 
    \exp \bigl[\frac{-t^2}{4e(1+\beta)\ell(w)(n-\ell(w)+1)}\bigr].
  \end{equation*}
  Moreover, for any $a\in A$ and any $n> \frac{|A|+1}{tq(w)} + \ell(w)$
  we have
    \begin{align*}
  \P\bigl(|\hat{q}_n(a|w)-&q(a|w)|>t \bigl)\;\leq\notag\\
  &(|A|+1)\, \ee \exp \bigl[- (n\!-\!\ell(w)\!+\!1)\;\frac{[t-\frac{|A|+1}{(n\!-\!\ell(w)\!+\!1)q(w)}]^2 [q(w)+\frac{|A|}{n\!-\!\ell(w)\!+\!1}]^2}{16e(1+\beta)|A|^2(\ell(w)+1)}\bigl].
\end{align*}
\end{lemma}

\begin{proof}
Observe that for any finite sequence 
$w_1^j\in A^j$ 
\begin{equation*}
N_n(w_1^j) \;= \;\sum_{t=0}^{n-j}\; \prod_{i=1}^{j}\;[\,\Id_{\{X_{t+i}=w_i\}}\Id_{\{\xi_{t+i}=0\}}\,+\,
\Id_{\{X_{t+i}=\bar{w_i}\}}\Id_{\{\xi_{t+i}=1\}}\,].
\end{equation*}  
Define the process $\{U_t\colon t\in\Z\}$ by  
\begin{equation*}
U_t \;=\; \prod_{i=1}^{j}\;[\,\Id_{\{X_{t+i-1}=w_i\}}\Id_{\{\xi_{t+i-1}=0\}}\,+\,
\Id_{\{X_{t+i-1}=\bar{w_i}\}}\Id_{\{\xi_{t+i-1}=1\}}\,] - q(w_1^j)
\end{equation*}
and denote by $\MM_i$ the $\sigma$-algebra generated by $U_1,\dotsc,U_i$.
Applying Proposition 4 in \citet{DD} we obtain that, for any $r\geq 2$
\begin{align*}
\Vert N_n(w_1^j&)-(n-j+1)\/q(w_1^j)\Vert_r\\
&\leq\; \Bigl( \,2r \, \sum_{i=1}^{n-j+1} \,\max_{i\leq\ell\leq n-j+1}\Vert U_i \/ 
\sum_{k=i}^\ell \E(U_k | \MM_i) \Vert_\frac{r}2\Bigr)^{\frac12}\\
&\leq\; \Bigl( \,2r \, \sum_{i=1}^{n-j+1} \Vert 
U_i\Vert_\frac{r}2 \/ \sum_{k=i}^{n-j+1} \Vert \E(U_k | \MM_i) 
\Vert_\infty\Bigr)^{\frac12}\\
&\leq\;\Bigl( \,2r\, \sum_{i=1}^{n-j+1}\sum_{k=i}^{n-j+1} 
\sup_{\sigma_1^i\in A^{i}}
|\E(U_k | U_1^i = \sigma_1^i)|\Bigr)^{\frac12}\\
&\leq\;\Bigl( \,2r\, \sum_{i=1}^{n-j+1}\sum_{k=i}^{n-j+1} 
\sup_{x_1^i,\,\theta_1^i\in A^{i}}
|\E(U_k | X_1^i = x_1^i,\xi_1^i=\theta_1^i)|\Bigr)^{\frac12}\\
&\leq\;\Bigl( \,2r\, \sum_{i=1}^{n-j+1}\sum_{k=i}^{n-j+1} 
\sup_{x_1^i,\,\theta_1^i\in A^{i}}
|\P(Z_k^{k+j-1}=w_1^j | X_1^i=x_1^i,\xi_1^i=\theta_1^i)-q(w_1^j)|\Bigr)^{\frac12}
\end{align*}
Using Lemma~\ref{mixing} we can bound above the last expression by
\begin{equation*}
[2r(1+\beta)\ell(w)(n-j+1)]^\frac12.
\end{equation*}
Then, as in \citet{galves2007} we obtain 
\begin{equation*}
  \P(\,|N_n(w)-(n-\ell(w)+1)q(w)|\,>\,t\,)\,\leq \,e^{\frac1e} 
    \exp \bigl[\frac{-t^2}{4e(1+\beta)\ell(w)(n-\ell(w)+1)}\bigr]
\end{equation*}
and
\begin{align*}
  \P\bigl(|\hat{q}_n(a|w)-&q(a|w)|>t \bigl)\;\leq\notag\\
  &(|A|+1)\, \ee \exp \bigl[- (n\!-\!\ell(w)\!+\!1)\;\frac{[t-\frac{|A|+1}{(n\!-\!\ell(w)\!+\!1)q(w)}]^2 [q(w)+\frac{|A|}{n\!-\!\ell(w)\!+\!1}]^2}{16e(1+\beta)|A|^2(\ell(w)+1)}\bigl].
\end{align*}
    This concludes the proof of Lemma~\ref{estim1}
\end{proof}

\begin{lemma}\label{lemOn}
For any $\delta > 2(1+\frac{4\beta\beta^*}{\alpha})\epsilon$, for any 
\[
n\; > \; \frac{2(|A|+1)}{(\frac\delta2-\epsilon\,(1+\frac{4\beta\beta^*}{\alpha}) )q_d}+d
\]
and for any $w \in \tau$, $uw\in\hat{\tau}_n^{\delta,d}$
we have that
  \begin{align*}
  \P( \Delta_n(uw) >\delta  )\;\leq \;   2\, |A|\,(|A|+1) \,\ee \exp \bigl[- (n\!-\!d)\;\frac{[\frac{\delta}{2}- \epsilon\,(1+\frac{4\beta\beta^*}{\alpha})]^2 q_d^2}{32e(1+\beta)|A|^2(d+1)}\bigl].
  \end{align*}
  \end{lemma}

\begin{proof}
 Recall that 
  \[
  \Delta_n(uw) = \max_{a\in A} |\hat{q}_n(a|uw)-\hat{q}_n(a|\suf(uw))|.
  \]  
  Note that the fact $w\in\tau$ implies that
  for any finite sequence $u$ and any symbol
   $a\in A$ we have $p(a|uw)=p(a|\suf(uw))$.
  Hence,
  \begin{align*}
     |\hat{q}_n(a|uw)-\hat{q}_n(a|\suf(uw))| \;\leq \;\;&
      |\hat{q}_n(a|uw)-q(a|uw)| \; + \;  |q(a|uw)-p(a|uw)| \\
      &+\; |q(a|\suf(uw))-p(a|\suf(uw))|\\
      & +\;  |\hat{p}_n(a|\suf(uw))-q(a|\suf(uw))|.
  \end{align*}
  Then, using Theorem~\ref{principal} we have that
  \begin{align*}
    \P(\Delta_n(uw) >\delta)\;\leq\; \sum_{a\in A}\,&\bigl[\,
    \P\bigl( |\hat{q}_n(a|uw)-q(a|uw)|>\frac{\delta}{2}- \epsilon\,(1+\frac{4\beta\beta^*}{\alpha})\bigr)\\
    & + \P\bigl( |\hat{q}_n(a|\suf(uw))-q(a|\suf(uw))|>  \frac{\delta}{2}- \epsilon\,(1+\frac{4\beta\beta^*}{\alpha})\bigr)\bigr].
  \end{align*}
  Now, for 
  \[
  n\; > \; \frac{2(|A|+1)}{(\frac\delta2-\epsilon\,(1+\frac{4\beta\beta^*}{\alpha}) )q_d}+d
  \]
   we can bound above the right hand side of the expression above using Lemma~\ref{estim1} by 
  \[
 2\, |A|\,(|A|+1) \,\ee \exp \bigl[- (n\!-\!d)\;\frac{[\frac{\delta}{2}- \epsilon\,(1+\frac{4\beta\beta^*}{\alpha})]^2 q_d^2}{32e(1+\beta)|A|^2(d+1)}\bigl].
   \]
\end{proof}

\begin{lemma}\label{lemUn}
 For any $d$ satisfying (\ref{d}), for any $\delta < D_d - 2\epsilon(1+\frac{4\beta\beta^*}{\alpha})$,
 for any 
 \[
n\; > \;  \frac{4(|A|+1)}{(D_d - 2 \epsilon\,(1+\frac{4\beta\beta^*}{\alpha}) -\delta )q_d}+d
\]
and for any $w\in\hat{\tau}_n^{\delta,d}$ with $\ell(w)<K$ we have that
  \[
  \P(\bigcap_{uw\in\tau|_d} \{\Delta_n(uw) \leq \delta\})\,\leq 
2\,(|A|+1)\, \ee \exp \bigl[- (n\!-\!d)\;\frac{[D_d-2(1+\frac{4\beta\beta^*}{\alpha})\epsilon-\delta]^2 q_d^2}{256e(1+\beta)|A|^2(d+1)}\bigl]. 
    \]
\end{lemma}

\begin{proof}
As $d$ satisfies (\ref{d}) we have that there exists a $\bar{uw}\in\tau|_d$ such that
  $\bar{uw}\in\tau$. Then 
  \begin{equation*}
   \P(\bigcap_{uw\in\tau|_d} \{\Delta_n(uw) \leq \delta\})\,\leq\,  \P(\Delta_n(\bar{uw}) \leq \delta).
  \end{equation*}
Observe that for  any $a\in A$,
  \begin{align*}
    |\hat{q}_n(a|\suf(\bar{uw}))-\hat{q}_n(a|\bar{uw})|\,\geq\,\,& |p(a|\suf(\bar{uw}))-p(a|\bar{uw})| -
    |\hat{q}_n(a|\suf(\bar{uw}))-q(a|\suf(\bar{uw}))| - \\ 
    & |\hat{q}_n(a|\bar{uw})-q(a|\bar{uw})| - |q(a|\suf(\bar{uw}))-p(a|\suf(\bar{uw}))| -\\
    &|q(a|\bar{uw})-p(a|\bar{uw})|.
  \end{align*}
   Hence, we have that for any $a\in A$ 
  \[
  \Delta_n(\bar{uw}) \, \geq\, D_d - 2\epsilon(1+\frac{4\beta\beta^*}{\alpha}) - |\hat{q}_n(a|\suf(\bar{uw}))-q(a|\suf(\bar{uw}))| - |\hat{q}_n(a|\bar{uw})-q(a|\bar{uw})|\/.
  \]
  Therefore,
  \begin{align*}
    \P(\Delta_n(\bar{uw}) \leq \delta)\;\leq\; &\P\bigl(\,\bigcap_{a\in A}
    \{\,|\hat{q}_n(a|\suf(\bar{uw}))-q(a|\suf(\bar{uw}))|\geq\frac{D_d-2\epsilon(1+\frac{4\beta\beta^*}{\alpha})-\delta}2\,\}\,\bigl) \\
    &  +\; \P\bigl(\,\bigcap_{a\in A}
    \{\,|\hat{q}_n(a|\bar{uw})-q(a|\bar{uw})|\geq\frac{D_d-2\epsilon(1+\frac{4\beta\beta^*}{\alpha})-\delta}2\,\}\,\bigl)\/.
  \end{align*}
  As $\delta < D_d-2\epsilon(1+\frac{4\beta\beta^*}{\alpha})$ and 
  \[
  n\; > \; \frac{4(|A|+1)}{(D_d - 2 \epsilon\,(1+\frac{4\beta\beta^*}{\alpha}) -\delta )q_d}+d
  \]
  we can use Lemma~\ref{estim1} to bound above the right hand side of the 
  inequality above by
  \[
  2\,(|A|+1)\, \ee \exp \bigl[- (n\!-\!d)\;\frac{[D_d-2(1+\frac{4\beta\beta^*}{\alpha})\epsilon-\delta]^2 q_d^2}{256e(1+\beta)|A|^2(d+1)}\bigl]. 
\]
This concludes the proof of Lemma~\ref{lemUn}
\end{proof}

Now we proceed with the proof of our main result.

\begin{proof}[Proof of Theorem~\ref{expobounded}]

Define
\[
O_{n,\delta}^{K,d}=\bigcup_{\substack{w \in \tau\\[.1cm]\ell(w)< K}}\bigcup_{uw\in\hat\tau_n^{\delta,d}} \{ \Delta_n(uw) >\delta \}
\/,
\]
and
\[
U_{n,\delta}^{K,d}=\bigcup_{\substack{w \in \hat\tau_n^{\delta,d}\\[.1cm]\ell(w) < K}}\bigcap_{uw\in\tau|_d} \{\Delta_n(uw) \leq \delta\}.
\]
Then, if  $d < n$ we have that
\[
\{\hat{\tau}_n^{\delta,d}|_K \neq \tau|_K\} = O_{n,\delta}^{K,d}\cup U_{n,\delta}^{K,d}.
\]
Using the definition of $O_{n,\delta}^{K,d}$ and $U_{n,\delta}^{K,d}$ we have that 
\[
\P(\hat{\tau}_n^{\delta,d}|_K \neq \tau|_K)\;\leq\; \sum_{\substack{w \in
    \tau\\ \ell(w)<K}}\sum_{uw\in\hat{\tau}_n^{\delta,d}} \P(  \Delta_n(uw) > \delta) + \sum_{\substack{w \in
    \hat\tau_n^{\delta,d}\\\ell(w)<K}}\P ( \bigcap_{uw\in\tau|_d}  \Delta_n(uw) \leq \delta ) .
\]
Applying Lemma~\ref{lemOn} and Lemma~\ref{lemUn} we obtain, for 
\[
n \;>\; \frac{4(|A|+1)}{[\min(\delta, D_d-\delta)-2\epsilon(1+\frac{4\beta\beta^*}{\alpha})]q_d
 } + d,
\]
 the inequality 
\begin{align*}
\P(\hat{\tau}_n^{\delta,d}|_K &\neq \tau|_K)\,\leq \, 
 4\,\ee\,(|A|+1)\,|A|^{d+1} \exp \bigl[- (n-d)\;\frac{[ \min(\delta, D_d-\delta)-2\epsilon(1+\frac{4\beta\beta^*}{\alpha})]^2 q_d^2}{256e(1+\beta)|A|^2(d+1)}\bigl]. 
\end{align*}
We conclude the proof of Theorem~\ref{expobounded}. 
\end{proof}

\begin{proof}[Proof of Corollary~\ref{main_cor}]
It follows from Theorem~\ref{expobounded}, using 
the first Borel-Cantelli Lemma and
the fact that the bounds for the error estimation of the context tree are summable in $n$ for a fixed $d$ satisfying (\ref{d}). 
\end{proof}

\bibliography{./referencias}  
\bibliographystyle{dcu}

\end{document}